\documentclass{article}
\usepackage{amsmath,amsthm,amssymb}
\usepackage[authoryear,math]{artratex}
\begin{document}
~~~~~~~~~~~~~~~\textbf{A multivariate version of Polya-Carlson theorem}

~~~~~~~~~~~~~~~~~~~~~~~~~~~~~~~~~~~~~~~~~~~~~Tianlong Yu
\\
\\
\textbf{Abstract}:
Polya-Carlson theorem asserts that if a power series with integer coefficients and convergence radius 1 can be extended holomorphically out of the unit disc, it must represent a rational function. In this note, we give a generalization of this result to multivariate case and give an application to rationality theorem about D-finite power series.
\\
\\
\\
\textbf{Introduction}:
\begin{theorem}[Polya-Carlson]([2][8])
Let $f(z)=\sum_{n=0}^{\infty}a_{n}z^n$ be a power series converging on the unit disc. Assume $a_n\in \mathbb{Z}$ for any n and f can be extended holomorphically out of the unit disc, then f represents a rational function. i.e. there exists polynomials P and Q with integer coefficients such that $f(z)=\frac{P(z)}{Q(z)}$.
\end{theorem}

If a power series with integer coefficients has convergence radius bigger than 1, then it must be a polynomial.

If the power series has convergence radius smaller than 1, one can find non-rational examples with integer coefficients which can be extended holomorphically out of their convergence domain([9]).

In the multivariate(dim=n) case, the convergence domain of a power series is a complete Reinhardt domain of holomorphy. Lelong([6]) showed that if the power series has integer coefficients and convergence domain containing the n-torus $T^n=\{(z_1,\dots,z_n)\in \mathbb{C}^n||z_j|=1,j=1,\dots,n\}$, then it must be a polynomial. If $T^n$ lies outside the convergence domain, then he found a power series with integer coefficients having its convergence domain as envelope of holomorphy.

If $T^n$ lies on the boundary of the convergence domain, we shall give a sufficient and nessesary condition for a power series with integer coefficients to be rational which can be viewed as a higher dimensional generalization of the theorem of Polya-Carlson. The following is our main theorem.

\begin{theorem}
Let $f(z_1,\dots,z_n)=\sum_{j_1\geq0,\dots,j_n\geq0}a_{j_1\dots j_n}z_1^{j_1}\dots z_n^{j_n}$ where $a_{j_1\dots j_n}\in \mathbb{Z}$. Assume that $T^n=\{(z_1,\dots,z_n)\in \mathbb{C}^n||z_j|=1,j=1,\dots,n\}$ lies on the boundary of the convergence domain $\Omega$ of f. Then f represents a rational fuction if and only if there exists positive integers $m_1,\cdots,m_{n-1}$such that for any $(\theta_1,\cdots,\theta_{n-1})$,there exists $z_0$ lying on the boundary of the unit disc and an open set U containing $(z_0,e^{i\theta_1}z_0^{m_1},\cdots,e^{i\theta_{n-1}}z_0^{m_{n-1}})$ and a holomorphic function F on U such that F=f on $\Omega \cap U$. 
\end{theorem}

\textbf{Remark1}:As can be illustrated by simple examples(e.g. $f(z,w)=\frac{\sum_{j=0}^{\infty}w^{j!}}{1-z}$), there are no straightforward ``dichotomy'' analogies of Polya-Carlson theorem in higher dimensions. Therefore, our theorem is just a criterion of rationality.

\textbf{Remark2}:When n=2, our theorem implies that f represents a rational function exactly when f can be extended holomorphically across \textit{one} point of $T^2$.

\textbf{Remark3}:There are also generalizations of Polya-Carlson theorem in higher dimensions([13]). However, at present, we do not know whether his conditions are nessesary for f to be rational or whether his theorem can imply ours.
\\
\\

As an application of our main theorem, we reprove the rationality theorem about D-finite power series(where `D' represents `Differential') established by Bell and Chen([1],[5]).

\begin{theorem}[Bell-Chen]
Let f be a D-finite power series with integer coefficients converging on the unit polydisc, then f represents a rational function.
\end{theorem}

\textbf{Preliminaries}:
\\
\\
We firstly recall some basic definitions[4].

\begin{definition}[Transfinite diameter]
Let K be a compact subset in $\mathbb{C}$, n be a positive integer. Let 

$V_n(K)=max\{\prod_{i<j}|z_i-z_j|:z_1,\dots,z_n\in K\}$

$d_n:=V_n(K)^{\frac{2}{n(n-1)}}$, we call $d(K):=lim_{n\to \infty}d_n$ the transfinite diameter of K.
\end{definition}

\begin{definition}[Cebysev constant]
Let K be a compact subset in $\mathbb{C}$, n be a positive integer. Let

$M_n(K)=inf\{max_{z\in K}|P(z)|:P(z)=z^n+\sum_{j=0}^{n-1}a_{j}z^j, a_0,\dots,a_{n-1}\in \mathbb{C}\}$

Let $\tau_n:=M_n(K)^{\frac{1}{n}}$, we call $\tau(K):=lim_{n\to \infty}\tau_n$ the Cebysev constant of K.
\end{definition}

It is well-known that $d(K)=\tau(K)$ for any K.
\\

\begin{definition}[D-finite power series]([7][12])
A formal power series $f(x_1,\dots,x_d)$ is called D-finite, if all the partial derivatives of f span a finite dimensional vector space over rational function field. Equivalently, for every $i\in \{1,\dots,d\}$, there exists a positive integer $m_i$ and polynomials $p_{i,j}, ~j\in\{1,\dots,m_i\}$ which are not all zero, such that f satisfies the following linear partial differential equations:

$p_{i,m_i}D_{x_i}^{m_i}f+\dots+p_{i,0}f=0,  i\in\{1,\dots,d\}$

\end{definition}

\textbf{Proof of the main theorem}:
\\

\begin{lemma}[Kronecker]([4])
Let $f(z)=\sum_{j=0}^{\infty}a_jz^j$ be a formal power series, where $a_j\in \mathbb{C}$.Then f is rational if and only if for any sufficiently large n we have $A_n=det(a_{i+j})_{i,j=0,\dots,n}=0$.
\end{lemma}

\begin{lemma}[Siciak 1962]([11])
Let $f(z_1,\dots,z_n)$ be a holomorphic function on a polydisc centered at $(0,\dots,0)$. Given positive integers $m_1,\dots,m_{n-1}$, let $g(t,\theta_1,\dots,\theta_{n-1})=f(t,e^{i\theta_1}t^{m_1},\dots,e^{i\theta_{n-1}}t^{m_{n-1}})$,
if there exists uncountable sets $E_j$ in $\mathbb{R},~ j=1,\dots,n-1$, such that $g(t,\theta_1,\dots,\theta_{n-1})$ is a rational function for $(\theta_1,\dots,\theta_{n-1}) \in E_1\times\dots\times E_{n-1}$, then f is rational.
\end{lemma}

In Siciak's work, $m_1=m_2=\dots=m_{n-1}=1$. The same proof works for the case here.

In the following proof of the main theorem, we follow the strategies used in the proof of the classical Szego's theorem in chapter 11 of [10].

Given $\varphi$, $\psi$ and $s>1$, such that $0<\varphi-\psi<2\pi$. For any $0\leq \delta<1$, define $\Gamma(\delta)$ be the following simple closed curve:
\begin{equation*}
\left\{ \begin{aligned}
	\text{arc:} \quad & (1-\delta)e^{it}, & t\in[\varphi,\psi+2\pi]; \\
	\text{line segment:} \quad & \rho e^{i\varphi}, & \rho\in[1-\delta,s]; \\
	\text{line segment:} \quad & \rho e^{i\psi} & \rho\in[1-\delta,s]; \\
	\text{arc:} \quad & se^{it} & t\in[\psi,\varphi].
\end{aligned} \right.
\end{equation*}

For the sake of convinience, denote $\iota(z)=\frac{1}{z}$ for nonzero z.

\begin{lemma} \label{Lemma:iota}
There exists $\delta_{0}>0$, such that for any $0<\delta<\delta_{0}$, we have $d(\iota(\Gamma(\delta))) < 1$.
\end{lemma}

\begin{proof}
By Runge's approximation theorem, there exists a polynomial
\begin{equation*}
	Q_{1}(z)=1+b_{k-1}z+\cdots+b_{0}z^{k},
\end{equation*}
such that $Q_{1}(z)<\frac{1}{2}$ when $z\in \Gamma(0)$ and $|z|=1$. Let $Q_{2}(z)=\dfrac{Q_{1}(z)}{z^{k}}$, there exists an open neighborhood $U$ of $\{|z|=1\}\cap \Gamma(0)$, such that$|Q_{2}|_{U}<\frac{1}{2}$. Let $r>1$ such that $\Gamma(0)\cap B_{r}(0)$ is contained in~$U$. Choose a positive integer $l$ such that $Q_{3}(z)=\dfrac{Q_{2}(z)}{z^{l}}$ satisfies
\begin{equation*}
	|Q_{3}(z)|<\frac{1}{2}, \quad \forall z\in \Gamma(0) \text{and } |z|>r.
\end{equation*}
Hence $|Q_{3}|_{\Gamma(0)}<\frac{1}{2}$. Choose an open neighborhood V of $\Gamma(0)$ such that $|Q_{3}|_{V}<\frac{1}{2}$, and choose $\delta_{0}>0$ such that~$\Gamma(\delta)$ is contained in~$V$, when $\delta<\delta_{0}$.Therefore we can find
\begin{equation*}
R(z) = c_{0}+c_{1}\frac{1}{z}+\cdots+c_{m-1}\frac {1}{z^{m-1}}+\frac{1}{z^{m}},
\end{equation*}
such that
\begin{equation*}
	\max_{z\in \Gamma(\delta)} |R(z)|< \frac{1}{2}, \quad \forall 0<\delta<\delta_0.
\end{equation*}
Hence for any positive integer $l$,
\begin{equation*}
M_{lm}(\iota(\Gamma(\delta)))<\frac{1}{2^{l}}.
\end{equation*}
Since $lim_{n\to \infty}\tau_n$ exists, we have $\tau(\iota(\Gamma(\delta)))\leq \frac{1}{2^{\frac{1}{m}}}$. This completes the proof because $d(\iota(\Gamma(\delta))) =\tau(\iota(\Gamma(\delta)))$.
\end{proof}

Now we just prove the main theorem in dimension two, because the proof of the general case is similar.

\begin{proof}
We follow the notations in Siciak's lemma. Let
\begin{equation*}
	P_{v}(z,w)=\sum_{j+nk=v}a_{jk}z^{j}w^{k},
\end{equation*}
then $f(z,w)=\sum_{v=0}^{\infty}P_{v}(z,w)$. Let
\begin{equation*}
	g_{\theta}(z)=f(z,e^{i\theta}z^{n})=\sum_{v=0}^{\infty}P_{v}(1,e^{i\theta})z^{v}.
\end{equation*}
For any positive integer $m$, denote
\begin{equation*}
	H_{m}(z) = \det(P_{v_{1}+v_{2}}(1,z))_{v_{1},v_{2}=0,1,\cdots,m}.
\end{equation*}

\textbf{Claim}: For$m$ large enough, we have$|H_{m}(z)|<1$ for any $z\in\partial D$.

Firstly suppose the claim holds. Since $a_{jk}$ are integers,  $H_{m}(z)$is a polynomial with integer coefficients. By the maximum principle for holomorphic functions,
$H_{m}(z)$ is identically 0. In particular, for any~$\theta$, $H_{m}(e^{i\theta})=0$. By Kronecker's lemma, for any~$\theta$, $g_{\theta}$~is a rational function. By Siciak's lemma, $f$is a rational function.
\\

We now prove the claim:
\\

For any $\theta_{0}$, there exists $z_{0}\in\partial D$ and an open neighborhood U of $(z_{0},e^{i\theta_{0}}z_{0}^{n})$, such that $f$ can be extended holomorphically to  $U$. Hence there exists $\epsilon_{0}>0$, and an open neighborhood $\Omega_{0}$ of $z_{0}$, such that for any $\theta_0-\epsilon_0<\theta<\theta_0+\epsilon_0$,  $g_{\theta}$ can be extended holomorphically from $D$ to $\Omega_{0}$. Following the notations in lemma \ref{Lemma:iota}, there exists~$\varphi,\psi,s>1,\delta>0$ such that the simple closed curve $\Gamma(\delta)$ is contained in $D\cup \Omega_{0}$, and $d(\iota(\Gamma(\delta)))<1$. By Cauchy's integral formula, for any $v\geq0$,
\begin{equation*}
P_{v}(1,e^{i\theta})=\frac{1}{2\pi i}\oint_{\Gamma(\delta)}g_{\theta}(z)\frac{1}{z^{v+1}}dz.
\end{equation*}
By the definition of determinant,
\begin{multline*}
H_{m}(e^{i\theta})
= \frac{1}{(2\pi i)^{m+1}}\oint_{\Gamma(\delta)}\cdots\oint_{\Gamma(\delta)} \frac{g_{\theta}(z_{1})}{z_{1}} \frac{g_\theta(z_2)}{z_{2}^{2}}\cdots\frac{g_{\theta}(z_{m+1})}{z_{m+1}^{m+1}} \\
\det\left(\frac{1}{z_{k}^{j-1}}\right)_{j,k=1,2,\cdots,m+1}dz_{1}\cdots dz_{m+1}.
\end{multline*}
If we take a permutation of the indices $1,\cdots,m+1$, the equality above still hold true. Take the sum of the equality above for all the permutations, we get
\begin{multline} 
(m+1)!H_{m}(e^{i\theta})=\frac{1}{(2\pi i)^{m+1}}\oint_{\Gamma(\delta)}\cdots\oint_{\Gamma(\delta)}g_{\theta}(z_{1})\cdots g_{\theta}(z_{m+1})\frac{1}{z_{1}\cdots z_{m+1}} \\
\left(\det\left(\frac{1}{z_{k}^{j-1}}\right)_{j,k=1,2,\cdots,m+1}\right)^{2}dz_{1}\cdots dz_{m+1}. \label{Eq:4-1}
\end{multline}
Note that
\begin{align*}
& \max_{z_{1},\cdots,z_{m+1}\in\Gamma(\delta)} \left(\det\left(\frac{1}{z_{k}^{j-1}}\right)_{j,k=1,2,\cdots,m+1}\right)^{2} \\
= & \max_{z_{1},\cdots,z_{m+1}\in\Gamma(\delta)} |\prod_{j<k}(\frac{1}{z_{j}}-\frac{1}{z_{k}})|^{2} \\
= & (V_{m+1}(\iota (\Gamma(\delta))))^{m(m+1)}.
\end{align*}
where we follow the notations in the definition of transfinite diameter. Since $d(\iota(\Gamma(\delta)))<1$, we have for m large enough, $V_{m+1}(\iota(\Gamma(\delta)))<\rho<1$. Let $L(\Gamma(\delta))$denotes the length of $\Gamma(\delta)$, $\eta=\min_{z\in
\Gamma(\delta)}|z|>0$, $M=sup\{|g_{\theta}(z)||z\in \Gamma(\delta), \theta_0-\epsilon_0<\theta<\theta_0+\epsilon_0\}<+\infty$. By(\ref{Eq:4-1}), for$m$large enough and $\theta_0-\epsilon_0<\theta<\theta_0+\epsilon_0$ we have
\begin{equation*}
|H_{m}(e^{i\theta})|<\frac{1}{(2\pi)^{m+1}}\frac{1}{(m+1)!}(L(\Gamma(\delta)))^{m+1}M^{m+1}\frac{1}{\eta
^{m+1}}\rho ^{m(m+1)}.
\end{equation*}
We can choose $m$ large enough, such that the right hand side of the inequality above is smaller than 1. Hence we can choose~$m_{0}>0$, such that when $m>m_{0}$, we have $|H_{m}(e^{i\theta})|<1$ for any $\theta$. This completes the proof of the claim.
\end{proof}

Now we give a new proof of Bell-Chen's rationality theorem.

\begin{lemma}([3])
\label{Lemma:ODEExistence}
Let$D$ denotes the disc of radius R centered at 0 in the complex plane, where~$R$>0~(can be infinity). Let~$A(z)$~ be a holomorphic $n\times n$ matrix on D, then for any~$w_{0}\in\mathbb{C}^n$, there exists an unique holomorphic vector-valued function~$w(z)$~on D,such that
\begin{equation} \label{Eq:1-1}
w'(z)=A(z)w(z) \quad (\forall z\in D),
\end{equation}
and
\begin{equation} \label{Eq:1-2}
w(0)=w_{0}.
\end{equation}
\end{lemma}

In the following, we just prove the 2-dim case of Bell-Chen's theorem, because the the higher dimensional case is paralled.

\begin{corollary} \label{Cor2}
Let $D$ be the unit disc on the complex plane, $S^1$ be its boundary. Let$f(z,w)$ be a power series with integer coefficients which is convergent on$D\times D$, and f(z,w) satisfies the following linear differential equation:
\begin{equation} \label{Eq:2-1}
p_{0}f+p_{1}f_{z}+\cdots+p_{r}f^{(r)}_{z}=0,
\end{equation}
\begin{equation} \label{Eq:2-2}
q_{0}f+q_{1}f_{w}+\cdots+q_{s}f^{(s)}_{w}=0.
\end{equation}
where $p_{j}$ $(j=1,\cdots,r)$ and $q_{k}$ $(k=1,\cdots,s)$ are holomorphic functions on a connected open neighborhood U of $S^1\times S^1$, and $p_{r}q_{s}$ is not identically 0 on U. Then$f$ is a rational function.
\end{corollary}

\begin{proof}
By the analycity of~$p_{r}$, there exists a positive integer $n$, such that for any~$\theta$, there exists $z_{0}\in S^1$ satisfying $p_{r}(z_{0},e^{i\theta}z_{0}^{n})\neq0$. Otherwise, for any~$w_{0}\in S^{1}$ and any positive integer~$n$,~$p_{r}(\cdot,w_{0})$ has uniformly distributed n zeros on~$S^{1}$. Since n and~$w_{0}$~are arbitary, $p_{r}$ is identically 0 on~$S^{1}\times S^{1}$. By the analycity of~$p_{r}$, we have $p_{r}$ is identically 0 on~$U$ , contradiction.

Let~$V\times V$ is a small bidisc centered at~$(z_{0},e^{i\theta}z_{0}^{n})$, such that~$p_{r}$ and $q_{s}$ are everywhere nonzero on~$V\times V$. Take~$w_{0}\in D\cap V$, such that there exists a small disc W centered at $w_{0}$ containing the point~$e^{i\theta}z_{0}^{n}$, and we can assume W is contained in~$V$. Rewrite~(\ref{Eq:2-2}) as the following ODE system depanding on z:
\begin{equation} \label{Eq:2-3}
\left\{
\begin{aligned}
y'_{1} & =y_{2}, \\
y'_{2} & =y_{3}, \\
& \cdots \\ 
y'_{s-1} & =y_{s}, \\
y'_{s} & =-\frac{1}{q_{s}} (q_{s-1}y_{s}+\cdots+q_{0}y_{1}), \\
y_{1}(z,w_{0}) & =f(z,w_{0}) , \cdots ,y_{s}(z,w_{0})=f^{(s-1)}_{w}(z,w_{0}).
\end{aligned}
\right.
\end{equation}

By lemma\ref{Lemma:ODEExistence} , for any $z\in D\cap V$,  (\ref{Eq:2-3}) has a solution~$g(z,w)$ on W, and by the uniqueness of the solution, $g(z,w)=f(z,w)$ when $w\in D\cap W$. According to the proof of lemma\ref{Lemma:ODEExistence},$g(z,w)$is holomorphic with respect to (z,w). Hence by the identity theorem of holomorphic functions, g(z,w) is a holomorphic extension of $f$, and $g(z,w)$ also satisfy~(\ref{Eq:2-1}), where~$z\in D\cap V$ and $w\in (D\cap V)\cup W$. Similarly, we can extend g to some open neighborhood of $(z_{0},e^{i\theta}z_{0}^{n})$ by solving the ODE system~(\ref{Eq:2-1}).

Therefore under the conditions of corollary~\ref{Cor2}, $f(z,w)$ satisfies the conditions of the main theorem. By our main theorem , we conclude that f is rational.
\end{proof}

\textbf{Acknowledgement}:

We would like to thank professor Xiangyu Zhou for inducing me to this interesting topic. We also thank professor Shaoshi Chen for many helpful discussions.

Tianlong Yu: National Center for Mathematics and Interdisciplinary Sciences, CAS, Beijing, 100190, China.

E-mail address: yutianlong940309@163.com
\end{document}